\documentclass[11pt,notitlepage,twoside,a4paper]{amsart}
 \usepackage{amsfonts}

\usepackage{amsmath,amssymb,enumerate}

\usepackage{epsfig,fancyhdr,color}

\usepackage{bigstrut}
\usepackage{amssymb}
\usepackage{amsmath,amsthm}
\usepackage{latexsym}
\usepackage{amscd}
\usepackage{psfrag}
\usepackage{graphicx}
\usepackage[latin1]{inputenc}
\usepackage[all]{xy}
\usepackage[mathcal]{eucal}

\definecolor{NoteColor}{rgb}{1,0,0}

\renewcommand{\textsc}{\textcolor{red}}

\newtheorem{theorem}{\rm\bf Theorem}[section]
\newtheorem{proposition}[theorem]{\rm\bf Proposition}
\newtheorem{lemma}[theorem]{\rm\bf Lemma}
\newtheorem{corollary}[theorem]{\rm\bf Corollary}
\newtheorem{theorem 1}{\rm\bf Proposition 1}
\newtheorem*{theorem 2}{\rm\bf Proposition 2}

\theoremstyle{definition}

\theoremstyle{remark}

\newtheorem{questions}[theorem]{\rm\bf Questions}

\def\interieur#1{\mathord{\mathop{\kern 0pt #1}\limits^\circ}}

\title[Mapping  class group]
{On the classification of mapping class actions on Thurston's asymmetric metric}

\author{L. Liu}
\address{Lixin Liu, Department of Mathematics, Sun Yat-Sen University, 510275, Guangzhou, P. R. China}
\email{mcsllx@mail.sysu.edu.cn}

\author{A. Papadopoulos}
\address{Athanase Papadopoulos,  Universit{\'e} de Strasbourg and CNRS,
7 rue Ren\'e Descartes,
 67084 Strasbourg Cedex, France} \email{athanase.papadopoulos@math.unistra.fr}
\date{\today}

\author{W.  Su}
\address{Weixu Su, Department of Mathematics, Fudan University, 200433, Shanghai, P. R. China, and Universit\'e de Strasbourg and CNRS, 7 rue Ren\'e Descartes,
67084 Strasbourg Cedex, France}
\email{suweixu@gmail.com}

\author{G. Th\'eret}
\address{Guillaume Th\'eret, Institut de math\'ematiques de Bourgogne, Universit\'e de Bourgogne UMR 5584 du CNRS BP 47870 21078 Dijon cedex,
France}
\email{guillaume.theret71@orange.fr}

\date{\today}

\thanks{}

\begin{document}

\begin{abstract}
We study the action of the elements of the mapping class group of a surface of finite type on the Teichm\"uller space of that surface equipped with Thurston's asymmetric metric. We classify such actions as elliptic, parabolic, hyperbolic and pseudo-hyperbolic, depending on whether the translation distance of such an element is zero or positive and whether the value of this translation distance is attained or not, and we relate these four types to Thurston's classification of mapping classes. The study is parallel to the one made by Bers in the setting of Teichm\"uller space equipped with Teichm\"uller's metric, and to the one made by Daskalopoulos and Wentworth in the setting of Teichm\"uller space equipped with the Weil-Petersson metric.
\end{abstract}

\maketitle


\medskip

\medskip

\tableofcontents

\section{Introduction}\label{intro}

Let $S=S_{g,n}$ be a connected oriented surface of finite type, of genus $g$ with $n$ punctures. We assume that the Euler characteristic of $S$ is negative. We shall consider hyperbolic structures on $S$ and any such structure will be complete and of finite area. We say that two complex structures (respectively hyperbolic structures)  $X$ and $Y$ on $S$ are equivalent  if there is a conformal map (respectively an isometry) from $X$ to $Y$ which is homotopic to the identity map of $S$.
The Teichm\"uller space $\mathcal{T}(S)$ of $S$ is the space of complex structures (or, equivalently, the space of hyperbolic structures) on $S$ up to equivalence. 

An \emph{asymmetric metric} on a set $M$ is a nonnegative function $\delta$ on 
$ M \times M$ which satisfies the axioms of a metric except the symmetry axiom, that is, we do not require that $\delta(x,y)=\delta(y,x)$ for all $x, y \in M$.

Thurston defined in \cite{Thurston2} an asymmetric metric $d_L$ on  $\mathcal{T}(S)$
by setting for each $X,Y \in \mathcal{T}(S)$
\begin{equation}\label{equ:Lip}
d_L(X,Y)=\inf_{f} \log L_f,
\end{equation}
where $X$ and $Y$ are considered as (equivalence classes of) complete finite area hyperbolic structures on $S$, where the infimum is take over all homeomorphsims $f:X \to Y$ which are homotopic to the identity map of $S$ and where 
 $L_f$ is the Lipschitz constant of $f$, that is,
 \[L_f=\sup_{x\not=y}\frac{d_Y(f(x),f(y))}{d_X(x,y)}\]
 for $x$ and $y$ in $X$.

 We shall call this asymmetric metric the \emph{Thurston asymmetric metric}. 

Thurston \cite{Thurston2} proved that  there is a (not necessarily unique) extremal Lipschitz homeomorphsim that realizes the infimum in (\ref{equ:Lip}), and that
$$d_L(X,Y)=\log \sup_{\gamma}  \frac{\ell_Y(\gamma)}{\ell_X(\gamma)},$$
where $\ell_X(\gamma)$ denotes the hyperbolic length of $\gamma$ in $X$ and $\gamma$ ranges over all essential (that is, neither homotopic to a point nor a puncture) simple closed curves on  $S$. We denote the Lipschitz constant of such an extremal Lipschitz map by $L(X,Y)$. Thurston showed that this function $d_L$ is indeed an asymmetric metric on Teichm\"uller space, that it is a Finsler metric  and that any two points in  Teichm\"uller space can be joined by a (not necessarily unique) $d_L$-geodesic (that is, a shortest path with respect to $d_L$). Thurston also made a relation between extremal Lipschitz maps and a class of $d_L$-geodesics called \emph{stretch lines}. (We shall recall the definition below.)

The Teichm\"uller metric is defined by 
$$d_T(X,Y)=\frac{1}{2}\inf_{f} \log K_f(X,Y),$$
where $K_f(X,Y)$ is the quasiconformal dilatation of a homeomorphsim $f:X \to Y$ homotopic to the identity map of $S$. Using a lemma of Wolpert, we have \cite{Wolpert}, $d_L  \leq 2d_T$.

By definition, the topology induced by Thurston's asymmetric metric on $\mathcal{T}(S)$ is the one induced by
 its  symmetrization, 
$$\frac{1}{2} \{d_L(X,Y)+ d_L(Y,X)\}$$
which is a genuine metic. This topology coincides with the one induced by
  the Teichm\"uller metric, see Li \cite{Li} and Papadopoulos-Th\'eret \cite{PT} for information about the symmetrization. In an asymmetric metric space $(M,\delta)$, the \emph{left} and \emph{right} open balls centered at a point $x$ and of radius $\epsilon$, defined respectively by
\[\{y\in M\ \vert \ \delta(x,y)<\epsilon\}\]
and 
\[\{y\in M\ \vert \ \delta(y,x)<\epsilon\}\]
 may have completely different behaviors. 
Endowed with Thurston's asymmetric metric, the space $\mathcal{T}(S)$  is complete, and the left and right closed balls are compact. See \cite{Liu2001}, \cite{PT} and \cite{PTHand} for the topology induced by this asymmetric metric.

The \emph{mapping class group} $\mathrm{Mod}(S)$ of $S$ is the group of homotopy classes of orientation-preserving homeomorphism of $S$. This group acts properly discontinuously and isometrically on $\mathcal{T}(S)$ endowed with Teichm\"uller's metric or with Thurston's asymmetric metric. The quotient space is the moduli space
$$\mathcal{M}(S)=\mathcal{T}(S)/ \mathrm{Mod}(S).$$

Thurston used his compactification of $\mathcal{T}(S)$ by the space of projective measured foliations to prove the following theorem (see Thurston \cite{Thurston}, Theorem 4 or \cite{FLP} for details).

\begin{theorem}[Nielsen-Thurston classification]  Each $\varphi\in \mathrm{Mod}(S)$ is either periodic, reducible or pseudo-Anosov. Pseudo-Anosov mapping classes are neither periodic nor reducible.

\end{theorem}
Here, $\varphi$ is said to be \emph{periodic} if it is of finite order; $\varphi$ is said to be \emph{reducible} 
if there is a multicurve (i. e. a union of disjoint essential curves) on $S$ which is globally fixed by $\varphi$; $\varphi$ is said to be \emph{pseudo-Anosov} if it preserves a pair of projective classes of transverse  measured foliations having no singular closed leaves, multiplying the transverse measure on one by some constant $K>1$ and on the other by the constant $1/K$. We shall recall this classification more precisely in the following sections.

\bigskip

Let $\varphi$ be an element of  $\mathrm{Mod}(S)$. We set 
\begin{equation}\label{equ:translation}
a(\varphi)=\inf_{X\in \mathcal{T}(S) }d_L(X,\varphi(X)).
\end{equation}
We call $a(\varphi)$ the \emph{translation distance}  of $\varphi$ with respect to Thurston's asymmetric metric. We say that $a(\varphi)$ is \emph{realized} if there exists some $X\in \mathcal{T}(S) $ such that 
$a(\varphi)=d_L(X,\varphi(X))$.

It is interesting to compare $a(\varphi)$ with the translation distance defined by the Teichm\"uller metric
$d_T$:
$$b(\varphi)=\inf_{X\in \mathcal{T}(S) }d_T(X,\varphi(X)).$$
Since $d_L  \leq 2d_T$, we have $a(\varphi)\leq 2b(\varphi)$.

Since Thurston's asymmetric metric is not symmetric, we may also consider
$$a(\varphi^{-1})=\inf_{X\in \mathcal{T}(S) }d_L(X,\varphi^{-1}(X))=\inf_{X\in \mathcal{T}(S)}d_L(\varphi(X),X).$$

To motivate the work done in this paper, we recall that Bers \cite{Bers} studied the translation distance $b(\varphi)$ and proved the following classification result.

\begin{theorem} [Bers \cite{Bers}] Let $\varphi\in \mathrm{Mod}(S)$. Then
\begin{enumerate}
\item \label{p1} $b(\varphi)=0 $ and realized $\Longleftrightarrow$ $\varphi$ is periodic.
\item  $b(\varphi)$ is not  realized $\Longleftrightarrow$ $\varphi$ is reducible.
\item  $b(\varphi)>0 $ and realized $\Longleftrightarrow$ $\varphi$ is pseudo-Anosov.
 \end{enumerate}
\end{theorem}

In case (\ref{p1}), saying that $\varphi$ is \emph{periodic} means that it is represented by a periodic homeomorphism of the surface.

Bers interpreted the realization of $b(\varphi)$ as an extremal problem. In the case where $b(\varphi)$ is not  realized, Bers further showed that there is a generalized solution of the extremal problem which is given by a noded Riemann surface, and that the restriction of $\varphi$ on each nonsingular component of the noded Riemann surface is either periodic or pseudo-Anosov.

 Daskalopoulos and Wentworth \cite{DW} obtained an analogous classification of elements of $\mathrm{Mod}(S)$ in terms of  the  Weil-Petersson translation distance.

In analogy with  Bers' classification of elements of $\mathrm{Mod}(S)$ with respect to the Teichm\"uller metric,  we classify $\varphi\in \mathrm{Mod}(S)$ into four types:

\bigskip
\begin{center}
\begin{tabular}{ | c | c | c | r |}
\hline
Elliptic  & Parabolic  & Hyperbolic  & Pseudo-hyperbolic  \\
\hline
$a(\varphi)=0$    & $a(\varphi)=0$   & $a(\varphi)>0$  & $a(\varphi)>0$  \\
and realized  & but not realized  & and realized  &  but not realized  \\
\hline
\end{tabular}
\end{center}

\bigskip

The aim of our investigation in this paper is to analyze each of these cases.

It is easy to see that $\varphi$ is elliptic if and only if $\varphi$ has a fixed point in $\mathcal{T}(S)$.  Indeed, if $\varphi$ has a fixed point in $\mathcal{T}(S)$, then there is a complex structure $X$ on $S$ such that $\varphi$ induces a biholomorphic automorphism of $X$. Using a classical result of Hurwitz, we deduce that $\varphi$ is of finite order, therefore periodic. Conversely, if $\varphi$ is periodic, Nielsen showed that $\varphi$ has a fixed point in $\mathcal{T}(S)$. In fact, Kerckhoff proved a stronger result, namely, that the action of every finite subgroup of $\mathrm{Mod}(S)$ has a fixed point in $\mathcal{T}(S)$.  (This was his solution to the so-called the \emph{Nielsen realization problem}.)

A stretch line is a $d_L$-geodesic of a special kind, defined by ``stretching" along a complete geodesic lamination. We shall recall the definition in  Section \ref{sec:pa}.

In this paper, we prove the following:

\begin{theorem} \label{thm:pa} 
If $\varphi$ is pseudo-Anosov, then it leaves a $d_L$-geodesic invariant. Furthermore, for some $n\in \mathbb{N}$, $\varphi^n$ leaves a stretch line invariant.
\end{theorem}

We also prove the following:

\begin{theorem} \label{thm:ab} 
If $\varphi$ is pseudo-Anosov, then $a(\varphi)=a(\varphi^{-1})=b(\varphi)$.
\end{theorem}

The projection on moduli space of the $\varphi$-invariant $d_L$-geodesic provided by Theorem \ref{thm:pa} is a  closed $d_L$-geodesic, and it is homotopic to a Teichm\"uller closed geodesic of the same length. 

In analogy with the case of the Teichm\"uller metric, we prove the following:
\begin{theorem} \label{thm:irreducible} 
If $\varphi$ is irreducible, then $\varphi$ is either elliptic or hyperbolic.
\end{theorem}

\begin{theorem} \label{thm:hyperbolic} 
If $\varphi$ is of infinite order and leaves a $d_L$-geodesic invariant, then $\varphi$ is hyperbolic. The translation distance $a(\varphi)$ is attained at any point on the geodesic.
\end{theorem}

\begin{theorem} \label{thm:parabolic} 
The mapping class $\varphi$ is of infinite order, reducible and periodic on all reduced components on $S$
if and only if  $\varphi$ is parabolic.
\end{theorem}

  To show that a reducible map $\varphi\in \mathrm{Mod}(S)$ cannot attain $b(\varphi)$ in Teichm\"uller space, Bers  \cite{Bers}  used the notion of Nielsen extension of Riemann surfaces to show that for any  $X\in\mathcal{T}(S)$, 
 one can ``shrink" the reduced curves to get a new Riemann surface $Y$ satisfying $d_T(Y, \varphi(Y))< d_T(X, \varphi(X))$. Bers' construction does not apply to the setting of Thurston's asymmetric metric. We shall see in \S \ref{s:red-hyp} of this paper that there exists a reducible map $\varphi$ which is hyperbolic with respect to the asymmetric $d_L$, and therefore for which $a(\varphi)$ is realized.

These results give quite a good picture of the behavior of the action of a mapping class with respect to Thurston's asymmetric metric, in terms of Thurston's classification. It remains to study the case where $\varphi$ is reducible and has at least one pseudo-Anosov  reduced component. 
 In the setting of the Teichm\"uller metric, to show that a hyperbolic map $\varphi\in \mathrm{Mod}(S)$ ($b(\varphi)>0 $ and realized) leaves a  Teichm\"uller
geodesic invariant, Bers \cite{Bers} used the fact that any two points in Teichm\"uller space are connected by a unique Teichm\"uller geodesic. Since Thurston's asymmetric metric is not  uniquely geodesic, Bers' argument does not work here.  In the paper \cite{DW}  by Daskalopoulos and Wentworth,  the strict convexity property of hyperbolic length functions along Weil-Petersson geodesics is used to show that a reducible map cannot leave a Weil-Petersson geodesic invariant. We don't know whether geodesic length functions are convex along  $d_L$-geodesics.

\begin{questions} 
\begin{enumerate}
We propose the following:

\item Is Theorem \ref{thm:ab} true in general ?

\item If $\varphi$ is reducible  and has at least one pseudo-Anosov  reduced component, is it pseudo-hyperbolic? Can it be hyperbolic?

\item If $\varphi$ is hyperbolic, does it leave a $d_L$-geodesic invariant?
is it pseudo-Anosov? If  $\varphi$ is of infinite order and leaves a $d_L$-geodesic (or even a stretch line) invariant, is it pseudo-Anosov?
\end{enumerate}
\end{questions}

In \S \ref{s:red-hyp}, we give an example, on any surface of genus $\geq 2$, of a reducible mapping class which is of hyperbolic type.

It is known that closed Teichm\"uller geodesics in moduli space correspond to conjugacy classes of pseudo-Anosov maps. The above questions are related to the question of whether closed $d_L$-geodesics in moduli space  correspond to conjugacy classes of pseudo-Anosov maps.

\section{Irreducible maps} 
In this section, we prove Theorems \ref{thm:irreducible} and \ref{thm:hyperbolic}. 

Let $\mathcal{S}$ be the set of homotopy classes of  essential simple closed curves on $S$.  A finite non-empty subset $\{C_1, \cdots, C_k\} \subset \mathcal{S}$ is called \emph{admissible} if $C_i \neq C_j$ for all $i \neq j$ and if $C_1, \cdots, C_k$ can be realized simultaneously as mutually disjoint curves. We call $\varphi\in \mathrm{Mod}(S)$ \emph{reducible} if there is an admissible set $\{C_1, \cdots, C_k\} $ such that
$$\varphi(\{C_1, \cdots, C_k\} )=\{C_1, \cdots, C_k\} .$$

An element $\varphi\in \mathrm{Mod}(S)$ is called \emph{irreducible} if there is no admissible set reduced by $\varphi$.

We shall use the following hyperbolic geometry estimate: Given a simple closed geodesic $\alpha$ of length $\ell_\alpha$ on a hyperbolic surface, there exists a collar neighborhood about $\alpha$ of width $2\omega_\alpha$, with 
$$\sinh \omega_\alpha \sinh (\ell_\alpha/2)=1.$$
See Buser \cite{Buser} for a reference.  It follows that for  small $\ell_\alpha$, the length of any simple closed curve that intersects $\alpha$ is large. The following consequence is well-known:
\begin{lemma} \label{lem:collar} 
There is a universal constant $\delta_0>0$ such that any two distinct simple closed geodesics on a hyperbolic surface with hyperbolic length less than $\delta_0$ are disjoint.
\end{lemma}

We also need the following theorem. See Bers \cite{Bers2} for a proof.
\begin{theorem}[Mumford's compactness theorem]
For each $\epsilon>0$,  the subset 
$$\mathcal{M}_{\epsilon}(S)=\{X\in \mathcal{M}(S):  \underline\ell(X)\geq \epsilon \}\subset \mathcal{M}(S)$$
is compact.
\end{theorem}
Here $\underline\ell(X) $ denotes the length of the shortest essential simple closed curve of a hyperbolic surface $X$. Note that this quantity is well-defined for elements $X$ in 
 $\mathcal{T}(S)$ or in $\mathcal{M}(S)$.

For each $\epsilon>0$, the subset $$\mathcal{T}_{\epsilon}(S)=\{X\in \mathcal{T}(S):  \underline\ell(X)\geq \epsilon \}$$
is called the \emph{$\epsilon$-thick part} of $ \mathcal{T}(S)$. Mumford's compactness theorem is equivalent to saying that for any sequence $(X_k)_{k\geq 1}$ in an $\epsilon$-thick part of $ \mathcal{T}(S)$, there is a subsequence of $(X_k)_{k\geq 1}$, still denoted by $(X_k)_{k\geq 1}$, and a sequence of $(\varphi_k)_{k\geq 1}$ in $ \mathrm{Mod}(S)$, such that $(\varphi_k(X_k))_{k\geq 1}$ converges in $ \mathcal{T}(S)$. 

\begin{lemma} \label{lem:systol} 
Let $\varphi\in \mathrm{Mod}(S)$ be irreducible and $X\in \mathcal{T}(S)$. Then, with $\delta_0$ being the constant given in Lemma \ref{lem:collar}, we have:
$$\left( L(X, \varphi(X))\right)^{3g-3+n} \geq \delta_0/\underline\ell(X).$$
\end{lemma}
\begin{proof}
Suppose that 
\begin{equation}\label{equ:systol} 
\left( L(X, \varphi(X))\right)^{3g-3+n} < \delta_0/\underline\ell(X).
\end{equation}
Let $\alpha\in  \mathcal{S}$ such that $\ell_X(\alpha)=\underline\ell(X)$. Consider the curves 
$\alpha_1=\alpha$, $\alpha_2=\varphi^{-1}(\alpha_1)$, $\cdots$, $\alpha_{3g-2+n}=\varphi^{-1}(\alpha_{3g-3+n})$. Since $$\ell_{X}(\varphi^{-1}(\alpha))=\ell_{\varphi(X)}(\alpha)\leq L(X, \varphi(X)) \ell_X(\alpha),$$
each $ \alpha_i$ satisfies
$$\ell_{X}(\alpha_i)\leq L(X, \varphi(X))^{i-1} \ell_X(\alpha).$$
By $(\ref{equ:systol} )$,  we have
$$\ell_X(\alpha_i) < \delta_0,  i=1, \cdots, 3g-2+n.$$
By Lemma \ref{lem:collar}, the curves $\alpha_1, \alpha_2, \cdots, \alpha_{3g-2+n}$ are mutually disjoint. Since there are at most $3g-3+n$ isotopy class of pairwise disjoint simple closed curves in $S$, there is a least number $1\leq j\leq 3g-3+n$ such that  $\alpha_j=\alpha_1$.

Note that $\{\alpha_1, \cdots, \alpha_{j-1}\}$  is admissible and $$\varphi(\{\alpha_1, \cdots, \alpha_{j-1}\})=\{\alpha_1, \cdots, \alpha_{j-1}\}.$$ As a result, $\varphi$ is reducible, which  contradicts  the assumption.
\end{proof}

\bigskip
\begin{proof}[Proof of Theorem \ref{thm:irreducible}]
By assumption, $\varphi$ is irreducible.  Consider a sequence $(X_k)_{k\geq 1}$ in $ \mathcal{T}(S)$ with 
\begin{equation}\label{equ:xk} 
\lim_{k\to \infty} d_L(X_k, \varphi(X_k))=a(\varphi).
\end{equation}
There is a constant $M>1$ such that for any $k \geq 1$, $L(X_k, \varphi(X_k))\leq M$.
By Lemma \ref{lem:systol}, 
$$\underline\ell(X_k)\geq \frac{\delta_0}{\left(L(X_k, \varphi(X_k))\right)^{3g-3+n}}\geq \frac{\delta_0}{M^{3g-3+n}}.$$
Therefore, the sequence $(X_k)$ lies in a thick part of $ \mathcal{T}(S)$.

Now we use a compactness argument of Bers \cite{Bers} to show that $a(\varphi)$ attains its infimum  in $ \mathcal{T}(S)$.

By Mumford's compactness theorem, there is a subsequence of $(X_k)_{k\geq 1}$, still denoted by $(X_k)_{k\geq 1}$, and a sequence $(\varphi_k)_{k\geq 1}$ in $\mathrm{Mod}(S)$, such that $(\varphi_k(X_k))_{k\geq 1}$ converges to some point $Y\in \mathcal{T}(S)$. Let $Y_k=\varphi_k(X_k)$.  Since $\mathrm{Mod}(S)$ acts isometrically with respect to Thurston's asymmetric metric, we have
$$d_L(X_k, \varphi(X_k))=d_L(\varphi_k(X_k), \varphi_k\circ \varphi(X_k))=d_L(Y_k, \varphi_k\circ \varphi \circ \varphi_k^{-1} (Y_k)).$$
By $(\ref{equ:xk} )$, we have
\begin{equation}\label{equ:yk} 
\lim_{k\to \infty} d_L(Y_k, \varphi_k\circ \varphi \circ \varphi_k^{-1} (Y_k))=a(\varphi).
\end{equation}

Since 
$$d_L(Y, \varphi_k\circ \varphi \circ \varphi_k^{-1} (Y_k))\leq d_L(Y, Y_k)+d_L(Y_k, \varphi_k\circ \varphi \circ \varphi_k^{-1} (Y_k))\leq  M_1 $$
for some sufficiently large constant $M_1$ and since $\mathcal{T}(S)$ is locally compact, up to a subsequence, we can assume that $(\varphi_k\circ \varphi \circ \varphi_k^{-1} (Y_k))_{k\geq 1}$ converges to some $Z \in \mathcal{T}(S)$.
Since $(Y_k)_{k\geq 1}$ converges to $Y$, we conclude that the sequence $(\varphi_k\circ \varphi \circ \varphi_k^{-1} (Y))_{k\geq 1}$ converges to $Z$.
For any $\epsilon>0$, there is a number $N>0$, such that for any $j,k >N$, 
$$d_L(\varphi_j\circ \varphi \circ \varphi_j^{-1} (Y), \varphi_k\circ \varphi \circ \varphi_k^{-1} (Y))< \epsilon.$$
Since $\mathrm{Mod}(S)$ acts properly discontinuously on $ \mathcal{T}(S)$, we have, for sufficiently large $k_0$, 
$$\varphi_k\circ \varphi \circ \varphi_k^{-1}=\varphi_{k_0}\circ \varphi \circ \varphi_{k_0}^{-1},  \ \forall \ k\geq k_0.$$
It follows from $(\ref{equ:yk})$ that 
$$d_L(Y,\varphi_{k_0}\circ \varphi \circ \varphi_{k_0}^{-1}(Y) )=a(f).$$
\end{proof}

\begin{corollary}\label{coro:reducible2}
Suppose that $\varphi\in \mathrm{Mod}(S)$ is parabolic or pseudo-hyperbolic and suppose that  $(X_k)_{k\geq 1}$ is a sequence in $ \mathcal{T}(S)$ with 
$$\lim_{k\to \infty} d_L(X_k, \varphi(X_k))=a(\varphi).$$
Then for any $\epsilon>0$, $X_k$ leaves 
$\mathcal{T}_{\epsilon}(S)$ for any sufficiently large $k$.
\end{corollary}
\begin{proof}
Otherwise, using Bers' compactness argument as we did in the proof of Theorem \ref{thm:irreducible}, we can show that $\varphi$ is elliptic or hyperbolic.
\end{proof}

\bigskip
\begin{proof}[Proof of Theorem \ref{thm:hyperbolic}]
The proof follows Bers' argument in \cite{Bers}, Theorem 5. We reproduce it here for the convenience of the reader. 
Suppose that  a $d_L$-geodesic through $X\in \mathcal{T}(S)$ is invariant under $\varphi$. 
Let $Y $ be any point in $\mathcal{T}(S)$. Since $\varphi$ is an isometry, we have
\begin{eqnarray*}
 nd_L(X, \varphi(X)) 
&=&d_L(X, \varphi(X))+ d_L(\varphi(X),\varphi^2(X))+  \cdots \\
&& + d_L(\varphi^{n-1}(X),\varphi^n(X)) \\
&=&d_L(X, \varphi^n(X)) \\
&\leq & d_L(X,Y)+ d_L(Y, \varphi^n(Y))+d_L(\varphi^n(Y),\varphi^n(X)) \\
&\leq&  d_L(X,Y)+d_L(Y, \varphi(Y)+ d_L(\varphi(Y),\varphi^2(Y))+  \cdots    \\ 
&&+ d_L(\varphi^{n-1}(Y),\varphi^n(Y))+d_L(\varphi^n(Y),\varphi^n(X))  \\
&=& d_L(X,Y)+ n d_L(Y, \varphi(Y))+d_L(Y,X).
\end{eqnarray*}
Since $n$ is arbitrary, $d_L(X, \varphi(X))\leq d_L(Y, \varphi(Y)$ for any $Y\in \mathcal{T}(S)$.
As a result, $a(f)=d_L(X, \varphi(X))$.
\end{proof}

Bers also proved that  $b(\varphi)>0$ is realized if and only if $\varphi$ leaves a Teichm\"uller geodesic invariant, cf. Bers \cite{Bers}, Theorem 5. By an analysis of  the Teichm\"uller geodesic left invariant by $\varphi$, he showed that  the initial and terminal quadratic differentials of the Teichm\"uller map in the homotopy class of $\varphi$  coincide.  It follows that $\varphi$ preserves a pair of transverse measured foliations, multiplying the transverse measure of one by $K$ (the dilation of the Teichm\"uller map) and of the other one by $1/K$. This means that  $\varphi$ is pseudo-Anosov. We will study pseudo-Anosov maps in Section \ref{sec:pa}.

\section{Laminations and Thurston's stretch maps} \label{sec:thurston} 
In this section, we  recall some necessary background on laminations and on Thurston's construction of strectch lines. 

Fix a hyperbolic structure on the surface $S$. 
A \emph{geodesic lamination} $\mu$ on $S$ is a closed subset of $S$ which is the union of disjoint simple geodesics (called the \emph{leaves} of $\mu$). A \emph{measured geodesic lamination} is a geodesic lamination $\mu$ with a transverse invariant measure of full support $\mu$.  If  a geodesic lamination $\mu$ has not all of its leaves going at both ends to cusps, then $\mu$ contains a compactly-supported sublamination admitting a transverse measure. 

The \emph{stump} of a geodesic lamination  $\mu$ is the maximal compact sublamination of $ \mu$ admitting a transverse measure. 

A geodesic lamination $\mu$ on a hyperbolic surface $X$ is \emph{complete} if its complementary regions are all isometric to ideal triangles. Associated with such a pair $(X, \mu)$ is a partial (that is, supported on a subsurface) measured foliation $F_{\mu}(X)$, satisfying the following:
\begin{enumerate}[(i)]
\item the foliation $F_{\mu}(X)$ intersects $\mu$ transversely and in each ideal triangle  of the complement of $\mu$, the leaves of $F_{\mu}(X)$  are composed of horocycles perpendicular to the boundary of the ideal triangle;
\item the non-foliated region in each ideal triangle is bounded by three pairwise tangent leaves, as in Figure \ref{horo};
\item  The transverse measure for $F_{\mu}(X)$ agrees with arc length on  $\mu$.

 \end{enumerate}

  \begin{figure}[ht!]
\centering
\psfrag{a}{\small  horocycle}
\psfrag{b}{\small horocycle of length 1}
\psfrag{c}{\small unfoliated region}
\includegraphics[width=.60\linewidth]{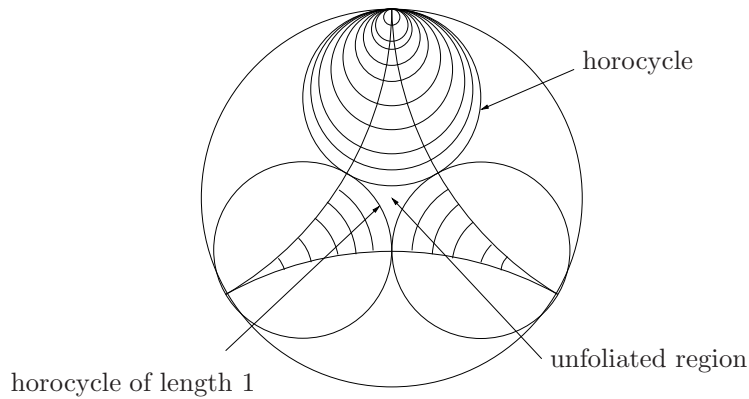}
\caption{\small {The horocyclic foliation.}}
\label{horo}
\end{figure}

By collapsing each non-foliated region of  $F_{\mu}(X)$ onto a tripod,  we get a measured foliation on $S$ of full support, still denoted by $F_{\mu}(X)$, which is well defined up to isotopy and which is called the \emph{ horocyclic foliation} corresponding to  $(X, \mu)$.  Note that since the hyperbolic structures we consider are complete, each puncture of $X$ has a neighborhood isometric to a cusp. If the set of punctures of $S$ is nonempty, each complete geodesic lamination on $S$ has some leaves going to cusps. The corresponding horocylic foliations are standard near the cusps, meaning that  in a neighborhood of each cusp, the leaves are circles  homotopic to the puncture, and the total transverse measure of an arc converging to the cusp is infinite.

We denote by $\mathcal{MF}(\mu)$ the space of equivalence classes of measured foliations that are transverse to $\mu$  and standard near the cusps.
 Thurston \cite{Thurston2} proved the following fundamental result.
\begin{theorem}\label{thm:thurston1}
The map $\mathbf{\Phi_\mu}:  \mathcal{T} (S) \to \mathcal{MF}(\mu): X \to F_\mu(X)$ is a homeomorphism.
\end{theorem}

We also recall the following definition from Thurston  \cite{Thurston2}.  A geodesic lamination $\lambda$ is \emph{totally transverse}  to $\mu$ if each leaf of $\lambda$ intersects $\mu$ transversely infinitely often and if each leaf of $\mu$ which does not go to a cups intersects $\lambda$ transversely infinitely often. (In counting intersections, the leaves are parametrized by the reals. In particular, with this convention, a simple closed geodesic meets any transverse arc  infinitely often.)

By a result of Thuston \cite{Thurston2}, $\mathcal{MF}(\mu)$ can be identified with the space of  (equivalence classes) of measured laminations with compact support and totally transverse to $\mu$. 

From now on, we always assume that a measured geodesic lamination has compact support.

We shall say that a measured geodesic lamination $\lambda$ is \emph{supported by} a geodesic lamination $\mu$ if $\lambda$ is a sublamination of $\mu$.
 
We fix a complete geodesic lamination $\mu$ on $S$. 

The  \emph{stretch line} directed by $ \mu$ and passing through $X\in  \mathcal{T} (S)$ is  the (image of a) path $$\mathbb{R} \ni t \to X_t= \mathbf{\Phi_\mu}^{-1}(e^t F_\mu(X)).$$
We shall call a segment of a stretch line a \emph{ stretch path}.
Stretch lines are  geodesics for the Thurston asymmetric metric. More precisely, suppose that $\mu$ supports a measured geodesic lamination $\lambda$.  Thus, for any two points $X_s, X_t, s\leq t$ on the stretch line, the Lipschitz distance $d_L(X_s, X_t)$ is equal to $t-s$ and realized by 
$$\log  \frac{\ell_{X_t} (\lambda)}{\ell_{X_s} (\lambda)}.$$

In fact, laminations that realize the Lipschitz distance $d_L(X,Y)$ (or, equivalently, that maximize the ratio of length) have been  characterized by Thurston. A geodesic lamination $\mu$ is \emph{ ratio-maximizing} for $X$ and $ Y$ if there is an $L(X,Y)$-Lipschitz map (homotopic to the identity map) from a neighborhood of $\mu$  in $X$ to a neighborhood of $\mu$ in $Y$. A geodesic lamination $\mu$ is called \emph{ chain-recurrent} if for any $\epsilon>0$ and for any $p\in \mu$ there is a closed $\epsilon$-trajectory of $\mu$ through $p$, that is, a closed unit speed path in the surface such that for any interval of length $1$ on the path there is an interval of length $1$ on some leaf of $\mu$ such that the two paths remain within $\epsilon$ of each other in the $C^1$ sense. 
By a result of Thurston \cite{Thurston2}, for any  two distinct  points $X, Y\in \mathcal{T} (S)$, there is a unique maximal (in the sense of inclusion) ratio-maximizing chain-recurrent lamination $\mu(X,Y)$ which contains all other ratio-maximizing chain-recurrent laminations.

With the above definitions, we can state the following theorem of Thurston.
\begin{theorem}[Thurston \cite{Thurston2}] \label{thm:thurston2}
Any two points $X,Y \in \mathcal{T} (S)$ can be connected by a $d_L$-geodesic which consists of a finite concatenation of stretch paths. Moreover, each of these stretch paths stretches along some complete lamination containing the unique ratio-maximizing chain-recurrent lamination $\mu(X,Y)$.
\end{theorem}

Since each element $\varphi$ of $\mathrm{Mod}(S)$ acts isometrically on $(\mathcal{T} (S), d_L)$, it maps $d_L$-geodesics to $d_L$-geodesics. We have the following:

\begin{lemma} Any $\varphi\in \mathrm{Mod}(S)$ maps a stretch line $ \mathbf{\Phi_\mu}^{-1}(e^t F_\mu(X))$ through $X$ to a stretch line through $\varphi(X)$, given by 
$$ \mathbf{\Phi_{\varphi(\mu)}}^{-1}(e^t F_{\varphi(\mu)}(\varphi(X))).$$
\end{lemma}
\begin{proof}
It is easy to see that the image of $ \mathbf{\Phi_\mu}^{-1}(e^t F_\mu(X))$ under $\varphi$ is 
a stretch line which passes through $\varphi(X)$, stretched along the complete lamination $\varphi(\mu)$ and with corresponding  horocyclic measured foliation  $\varphi(F_\mu(X))$.
Since the horocyclic measured foliation is uniquely determined by $\varphi(X)$ and $\varphi(\mu)$, and since $ \mathbf{\Phi_{\varphi(\mu)}}^{-1}(e^t F_{\varphi(\mu)}(\varphi(X)))$ is also a stretch line through $\varphi(X)$ directed by   $\varphi(\mu)$, it follows that $\varphi(F_\mu(X))=F_{\varphi(\mu)}(\varphi(X))$.
\end{proof}

We denote by $[\lambda]$ the projective class of a measured lamination (or measured foliation) $\lambda$. Such a $[\lambda]$ determines a element in Thurston's boundary of Teichm\"uller space. We refer to Thurston \cite{Thurston}  and \cite{FLP}  for definitions and properties of Thurston's compactification and Thurston's boundary. Here we just recall that Thurston's compactification is homeomorphic to a closed $(6g-6+2n)$-dimensional ball whose boundary is identified with the space of projective measured foliations and that the mapping class group action on Teichm\"uller space extends continuously to Thurston's compactification.

The following two theorems are important for our study.

\begin{theorem}[Papadopoulos \cite{Papa}] \label{thm:papa}
The stretch line  $ \mathbf{\Phi_\mu}^{-1}(e^t F_\mu(X))$ converges in the positive direction to  $[F_\mu(X)]$ in Thurston's boundary.

\end{theorem}

Recall that a measured geodesic lamination is said to be \emph{uniquely ergodic} if its transverse measure is unique up to scalar multiples.

\begin{theorem}[Th\'eret \cite{Theret}] \label{thm:gt}
Suppose that the stump $\lambda$ of the geodesic lamination $\mu$ is non-empty and is uniquely ergodic. Then $ \mathbf{\Phi_\mu}^{-1}(e^t F_\mu(X))$ converges in the negative direction to  $[\lambda]$ in Thurston's boundary.

\end{theorem}

Now consider a pair of totally transverse measured geodesic laminations $\lambda_1$ and $\lambda_2$. Suppose that $\lambda_1$ is uniquely ergodic. We can choose a complete geodesic lamination   $\mu$ whose stump is  $\lambda_1$. By Theorem \ref{thm:thurston1}, there exists a unique $X\in \mathcal{T} (S)$ such that $F_{\mu}(X)$ is equivalent to $\lambda_2$. (As we noted above, we identify $F_{\mu}(X)$ with $\lambda_2$.) It follows from 
Theorem  \ref{thm:papa} and Theorem  \ref{thm:gt} that the stretch line $ \mathbf{\Phi_\mu}^{-1}(e^t F_\mu(X))$ converges in the positive direction to $[\lambda_2]$ and converges in the negative direction to $[\lambda_1]$. 

We have the following consequence.

\begin{corollary}\label{thm:converge}
Let $\lambda_1$ and $\lambda_2$ be a pair of totally transverse measured geodesic laminations and suppose that $\lambda_1$ is uniquely ergodic.  Then there exists a stretch line which converges in the positive direction to $[\lambda_2]$ and in the negative direction to $[\lambda_1]$. 
\end{corollary}
In general, the stretch line in Corollary \ref{thm:converge}  is not unique.    If $\lambda_1$ is not maximal, there are  different completions of $\lambda_1$ which give different stretch lines in Teichm\"uller space.

\section{Pseudo-Anosov maps} \label{sec:pa} 

In this section, we prove Theorem \ref{thm:pa} and \ref{thm:ab}. 

Recall that an element $\varphi\in \mathrm{Mod}(S)$ is said to be \emph{pseudo-Anosov} if it has a representative, still denoted by $\varphi$, which is a homeomorphism of $S$ that preserves the projective classes of a pair of transverse measured foliations $(F_1, \lambda_1)$ and $(F_2, \lambda_2)$ and if there exists a constant $K>1$ such that 
$$\varphi((F_1, \lambda_1))=(F_1, \frac{1}{K}\lambda_1),$$
$$\varphi((F_2, \lambda_2))=(F_2, K\lambda_2).$$
The measured foliations $\lambda_1$ and $\lambda_2$ are called the \emph{stable} and \emph{unstable}  foliations of $\varphi$ respectively. The quantity $\log K$ is the \emph{topological entropy} of $\varphi$ \cite{Thurston} \cite{FLP}.

We recall some more terminology:

A measured foliation $\lambda$ is  \emph{minimal} if no closed curve in $S$ can be realized as a leaf  of $\lambda$. Equivalently, after Whitehead moves, the foliation has only dense leaves on $S$. Two measured foliations are \emph{topologically equivalent}  if after isotopy and  Whitehead moves, the leaf structure of the two foliations is the same. A theorem of Masur \cite{Masur} states that if a measured foliation $\lambda$ is minimal and uniquely ergodic, then for any other measured foliation $\lambda'$  satisfying $i(\lambda, \lambda')=0$ we have $\lambda'=c \lambda$ for some constant $c > 0$.

\begin{proposition}[Thurston, see \cite{FLP}]\label{lemma:ergodic}
Suppose that $\varphi\in \mathrm{Mod}(S)$ is pseudo-Anosov. Then each of the invariant foliations of $\varphi$ is minimal and uniquely ergodic.
\end{proposition}
In what follows, we shall often pass from the measured foliations $\lambda_1$ and $\lambda_2$ to the measured laminations that represent them and vice versa.

We endow $S$ with a hyperbolic structure and we identify $\lambda_1$ and $\lambda_2$ with the corresponding measured geodesic laminations. 
Since $\lambda_1$ is minimal and uniquely ergodic,  each connected component of $S\setminus \lambda_1$ is isometric to an ideal polygon or an ideal polygon with one puncture.
It follows that there are finitely many different complete geodesic laminations with stump $\lambda_1$.

\begin{lemma}
Let $\mu$ be a complete geodesic lamination with stump $\lambda_1$. Then $\mu$ and $\lambda_2$ are totally transverse.
\end{lemma}
\begin{proof}
By Proposition \ref{lemma:ergodic}, both $\lambda_1$ and $\lambda_2$ are minimal and uniquely ergodic. We replace $\lambda_1$ and $\lambda_2$ by their respective measured foliations. Up to Whitehead moves, we may assume that for both $\lambda_1$ and $\lambda_2$, each half-leaf is dense in the surface. Then it is easy to see that
each  leaf of $\lambda_2$ intersects $\mu$ transversely infinitely often, and each leaf of $\mu$ which does not go to a cusp intersects $\lambda_2$ transversely infinitely often. It follows from the definition that $\lambda_2$ is totally transverse to $\mu$.
\end{proof}

\begin{theorem}\label{thm:pa-stretch}
Suppose that $\varphi\in \mathrm{Mod}(S)$ is pseudo-Anosov.  Then  there exists an integer $n\geq 1$ such that  $\varphi^n$ leaves a stretch line invariant. 
\end{theorem}
\begin{proof}
Let $\lambda_1$ and $\lambda_2$ be the stable and unstable measured laminations of $\varphi$. 
 Choose a complete  geodesic lamination $\mu$ with stump $\lambda_1$. 
By the discussion before Corollary \ref{thm:converge}, there is a unique stretch line, denoted by $r(\mu, \lambda_2)$, which converges in the positive direction to  $[\lambda_2]$  and converges in the negative direction to  $[\lambda_1]$. 

Since $\varphi$ fixes $[\lambda_1]$ and $[\lambda_2]$  on Thurston's boundary, each $\varphi^k(r(\mu, \lambda_2))$ is a  stretch line which also converges in the positive direction to  $[\lambda_2]$  and converges in the negative direction to  $[\lambda_1]$.  

A stretch line satisfying the above conditions is uniquely determined by $\lambda_2$ and the completion of $\lambda_1$. 
Since there are finitely many different completions of $\lambda_1$, there is a integer $n$ such that $\varphi^n$ leaves $r(\mu, \lambda_2)$ invariant.
\end{proof}
By Theorem \ref{thm:hyperbolic}, $\varphi^n$ is hyperbolic. The translation distance $a(\varphi^n)$ attains its infimum at any point on $r(\mu, \lambda_2)$.

\begin{proposition}\label{pro:stretch}
With the above notations, 
$a(\varphi^n)=n\log K$.
\end{proposition}
\begin{proof}
Since $\varphi^n(\lambda_1)=\frac{1}{K^n} \lambda_1$,  for any $X\in  \mathcal{T} (S)$, 
$$ \ell_{\varphi^n(X)}(\lambda_1)=\ell_X(\varphi^{-n}(\lambda_1))=K^n \ell_X(\lambda_1).$$
Thus $a(\varphi^n) \geq n \log K$.

Take $X$ on the stretch line $r(\mu, \lambda_2)$.  Since the measure of $\lambda_2$ is ``stretched" by $K^n$, by definition of stretch line, the length of the measured lamination  $\lambda_1$  is also ``stretched" by $K^n$. We have $n\log K=d_L(X, \varphi^n(X))\geq a(\varphi)$.

\end{proof}

The proof of Theorem \ref{thm:pa-stretch} gives the following:

\begin{theorem}\label{thm:pa-stretch2}
Suppose that $\varphi\in \mathrm{Mod}(S)$ is a pseudo-Anosov mapping class with stable and unstable measured laminations $\lambda_1$ and $\lambda_2$ and suppose that $\lambda_1$ is complete.  Then  $\varphi$ leaves a unique stretch line invariant. 
\end{theorem}

In the general case, we don't know whether a pseudo-Anosov mapping class leaves a stretch line invariant. We will show that such a mapping class leaves a $d_L$-geodesic invariant. First notice the following property: 
\begin{proposition}\label{pro:stretch2}
Suppose that $\varphi\in \mathrm{Mod}(S)$ is pseudo-Anosov with topological entropy $\log K$. Then $\varphi$ is hyperbolic and $a(\varphi)=\log K$.
\end{proposition}
\begin{proof}
By the above results, there is a power $\varphi^n$ which fixes a stretch line and with translation length $a(\varphi^n)=n\log K$. By the triangle inequality,
$$n\log K\leq d_L(X, \varphi^n(X))\leq n d_L(X, \varphi(X)).$$
It follows that $\log K \leq a(\varphi)$.

To show the reverse inequality, we use the following lemma.
\begin{lemma}[Papadopoulos \cite{Papa}]\label{lem:Papa}
There is a constant $C>0$ such that for any simple closed curve $\alpha$ and for any integer $m$,
$$i(\varphi^m(\lambda_2), \alpha) \leq \ell_{\varphi^m(X)}(\alpha)   \leq i(\varphi^m(\lambda_2), \alpha)+C.$$
\end{lemma}
By this lemma, we have
\begin{eqnarray*}
 \frac{\ell_{\varphi(X)}(\alpha) ) }{\ell_X(\alpha)}  &= & \frac{\ell_{\varphi^{m}(X)}(\varphi^{-1}(\alpha) ) }{\ell_{\varphi^m(X)}(\alpha)}   \\
&\leq&  \frac{i(\varphi^m(\lambda_2), \varphi^{-1}(\alpha))+C }{i(\varphi^m(\lambda_2), \alpha)} \\
&=& \frac{K^m i(\lambda_2, \varphi^{-1}(\alpha))+C }{K^mi(\lambda_2, \alpha)} \\
&=& \frac{K^{m+1} i(\lambda_2, \alpha)+C }{K^mi(\lambda_2, \alpha)}.
\end{eqnarray*}
By taking $m \to +\infty$, we have 
$$ \frac{\ell_{\varphi(X)}(\alpha) ) }{\ell_X(\alpha)} \leq K. $$
Since  $\alpha$ is arbitrary, we have 
$$d_L(X, \varphi(X)) \leq \log K.$$
 Moreover, the infimum of $a(\varphi)$  is attained at any point on the stretch line  invariant under $\varphi^n$.
\end{proof}

Now we show the following: 
\begin{theorem}\label{thm:pa-stretch4}
A pseudo-Anosov map $\varphi\in \mathrm{Mod}(S)$ leaves a $d_L$-geodesic invariant. 
\end{theorem}
\begin{proof}
By Theorem \ref{thm:pa-stretch}, there exists an integer $n$ such that  $\varphi^n$ leaves a stretch line $r_1$ invariant. Note that $r_i=\varphi(r_{i-1}), i=1, \cdots, n-1$
are also invariant by $\varphi^n$.  

Choose any point $X$ on $r_1$ and choose a geodesic segment between $X$ and $\varphi(X)$. Denote this geodesic segment by $R_1$. We construct a path $R$ in Teichm\"uller space by setting 
$$R=R_1 \cup R_2 \cdots \cup R_n \cup \cdots$$
where $r_k=\varphi(R_{k-1})$. We  also extend the path to the reverse direction by iteration of $\varphi^{-1}$. In this way, we get an infinite path $R$ in Teichm\"uller space which is invariant by the action of $\varphi$. 

We claim that $R$ is a geodesic. To see this, consider the segment on $R$ connecting $\varphi^{-kn}(X)$ and $\varphi^{kn}(X)$. The length of this segment is equal to $2kn$ times $d_L(X, \varphi(X))$. On the other hand, both $\varphi^{-kn}(X)$ and $\varphi^{kn}(X)$ lie on the stretch line $r_1$, and the Lipschitz distance between $\varphi^{-kn}(X)$ and $\varphi^{kn}(X)$ is equal to $2kn$ times $\log K$ (the entropy of $\varphi$). We have shown in Proposition  \ref{pro:stretch2} that $d_L(X, \varphi(X))=\log K$. As a result, the length of the above segment on $R$ connecting $\varphi^{-kn}(X)$ and $\varphi^{kn}(X)$ is equal to the Lipschitz  distance between $\varphi^{-kn}(X)$ and $\varphi^{kn}(X)$. Since $k$ is arbitrary, $R$ is a $d_L$-geodesic.

\end{proof}

Combining Theorem \ref{thm:pa-stretch} and Theorem \ref{thm:pa-stretch4}, we obtain Theorem \ref{thm:pa}.

Note that if $\varphi$ is pseudo-Anosov with entropy $\log K$, it was proved by Bers \cite{Bers} that $b(\varphi)=\log K$. It follows that $a(\varphi)=b(\varphi)$. From this, we deduce Theorem \ref{thm:ab}. 
The projection of a $\varphi$-invariant $d_L$-geodesic on the moduli space is a closed $d_L$-geodesic, which is homotopic to a  closed Teichm\"uller geodesic with the same length.

Note that by results of Th\'eret in \cite{Thesis},  any two $d_L$-geodesics left invariant by a pseudo-Anosov map are asymptotic, and  two $d_L$-geodesics left invariant by two distinct pseudo-Anosov maps are divergent.

\section{Reducible maps}
By Bers' classification, we have 

\begin{proposition}\label{pro:reducible}
An element $\varphi\in \mathrm{Mod}(S)$ is of infinite order and reducible if and only if $b(\varphi) $ is not realized. Moreover, such an element $\varphi$ has no pseudo-Anosov reduced component
 if and only if $b(\varphi)=0$.
\end{proposition}

Suppose that $\varphi\in \mathrm{Mod}(S)$ is of infinite order, reducible and has no pseudo-Anosov reduced component.
Since $a(\varphi)\leq 2b(\varphi)$ and $b(\varphi)=0$,  $\varphi$ is parabolic, i. e., $a(\varphi)=0$.

On the other hand, suppose that $\varphi$ is of infinite order, reducible and has at least one pseudo-Anosov reduced component. Let $Y$ be a pseudo-Anosov reduced component of $\varphi$. Up to taking a power of $\varphi$, we may assume that $\varphi$ maps $Y$ to $Y$. There is a pair of transverse measured foliations  $\lambda_1$ and $\lambda_2$ on the surface $Y$ such that 
$$\varphi(\lambda_1)=\frac{1}{K}\lambda_1,  \varphi(\lambda_2)=K \lambda_2,$$
for some constant $K>1$.
 
It follows that for any $X\in \mathcal{T}(S)$, 
$$d_L(X,\varphi(X) )\geq \log \frac{\ell_X(\varphi^{-1}(\lambda_1))}{\ell_X(\lambda_1)}= \log K.$$
As a result, $a(\varphi)\geq \log K$.

We deduce the following:
\begin{theorem}\label{thm:parabolic}
$\varphi$ is parabolic if and only if  $\varphi$ is infinite order, reducible and has no pseudo-Anosov reduced component.
\end{theorem}

Now suppose that $\varphi$ is pseudo-hyperbolic. 

By Theorem \ref{thm:parabolic}, $\varphi$ has at least one pseudo-Anosov reduced component. The maximum of the topological entropy on all of the pseudo-Anosov components gives a lower bound for $a(\varphi)$. In fact, from Bers' results, it follows that $a(\varphi)\geq b(\varphi)$.

By Theorem \ref{thm:irreducible}, $\varphi$ is reducible.
Assume that  $\varphi$ is reduced by a maximal admissible set $\{C_1, \cdots, C_k\} $. Up to a power of $\varphi$, we may assume that $\varphi$ fixes each reduced component. Note that any power of $\varphi$ is also pseudo-hyperbolic.

By definition, there exists a sequence $(X_k)_{k\geq 1}$ in  Teichm\"uller space such that 
$$d_L(X_k, \varphi(X_k))\to a(\varphi).$$
By Corollary \ref{coro:reducible2}, for any $\epsilon>0$, $X_k$  leaves the $\epsilon$-thick part of $\mathcal{T}(S)$ for $k$ sufficiently large. 
This means that  $\underline\ell( X_k)\to 0$.

 Let $Y$ be any pseudo-Anosov reduced component of $Y$. But assumption, $\varphi$ fixes $Y$. 

\begin{lemma}\label{lemm:p-h}
Let $(\alpha_k)_{k\geq 1}$ be a sequence of simple closed curves such that $\ell_{ X_k}(\alpha_k)\to 0$. For sufficiently large $k$, each $\alpha_k$ is not contained in the interior of $Y$.
\end{lemma}
\begin{proof}
Suppose not and take a subsequence  of  $(\alpha_k)_{k\geq 1}$, still denoted by $(\alpha_k)$, contained in the interior $Y$. Since $\ell_{ X_k}(\alpha_k)\to 0$ and $d_L(X_k, \varphi(X_k))$  is uniformly bounded above,  using the proof of Lemma \ref{lem:systol}, $\varphi$ is reducible restricted on $Y$. This contradicts  the assumption that $Y$ is a pseudo-Anosov component. 

\end{proof}

\section{A reducible and hyperbolic mapping class}\label{s:red-hyp}

In this section, we  prove the following
\begin{proposition} \label{prop:ex} On an arbitrary surface of genus $\geq 2$, we can find a reducible mapping class which is hyperbolic. 
\end{proposition}
\begin{proof}
Consider a pseudo-Anosov map $\varphi$ on a hyperbolic surface of genus $g$ with $n>0$ boundary components, which we denote by $\Sigma_{g,n}$ in order to distinguish it from the surface $S_{g,n}$ with punctures.  We assume that $\varphi$ does not permute the complementary components of its stable laminations. We take two copies of  this map $\varphi$, which we denote by $\varphi_1$ and $\varphi_2$, and we denote the two underlying surfaces by $\Sigma_1$ and $\Sigma_2$.  We glue $\Sigma_1$ and $\Sigma_2$ in a symmetric manner along their boundary components and we obtain a homeomorphism of a closed hyperbolic surface $S=S_{2g+n-1,0}$ of genus $2g+n-1$. We denote the homeomorphism $\varphi_1\cup \varphi_2$ by $\varphi$. The surface $S$ has an order-two symmetry $\iota$ which commutes with $\varphi$ (i. e. $\varphi\circ\iota =\iota\circ\varphi$) and which fixes the multicurve $\alpha=\partial \Sigma_1=\partial \Sigma_2\hookrightarrow S$.

For $j=1,2$, let $\lambda_j$ ad $\mu_j$ be the stable and unstable geodesic laminations of $\varphi_j$:

 \[\displaystyle 
\begin{cases}
 \varphi_j(\lambda_j)=k\lambda_j\\
\displaystyle \varphi_j(\mu_j)=\frac{1}{k}\mu_j,
 \end{cases}
\]
$\lambda>1$.
The homeomorphism $\varphi$ is reducible and with no Dehn twist on the reducing multicurve, and it fixes the measured geodesic laminations $\lambda=\lambda_1\cup\lambda_2$ and $\mu=\mu_1\cup \mu_2$; cf. Figure \ref{phi}.

  \begin{figure}[ht!]
\centering
\psfrag{a}{\small $\alpha$}
\psfrag{f}{\small $\varphi_1$}
\psfrag{g}{\small $\varphi_2$}
\psfrag{i}{\small $\iota$}
\includegraphics[width=.60\linewidth]{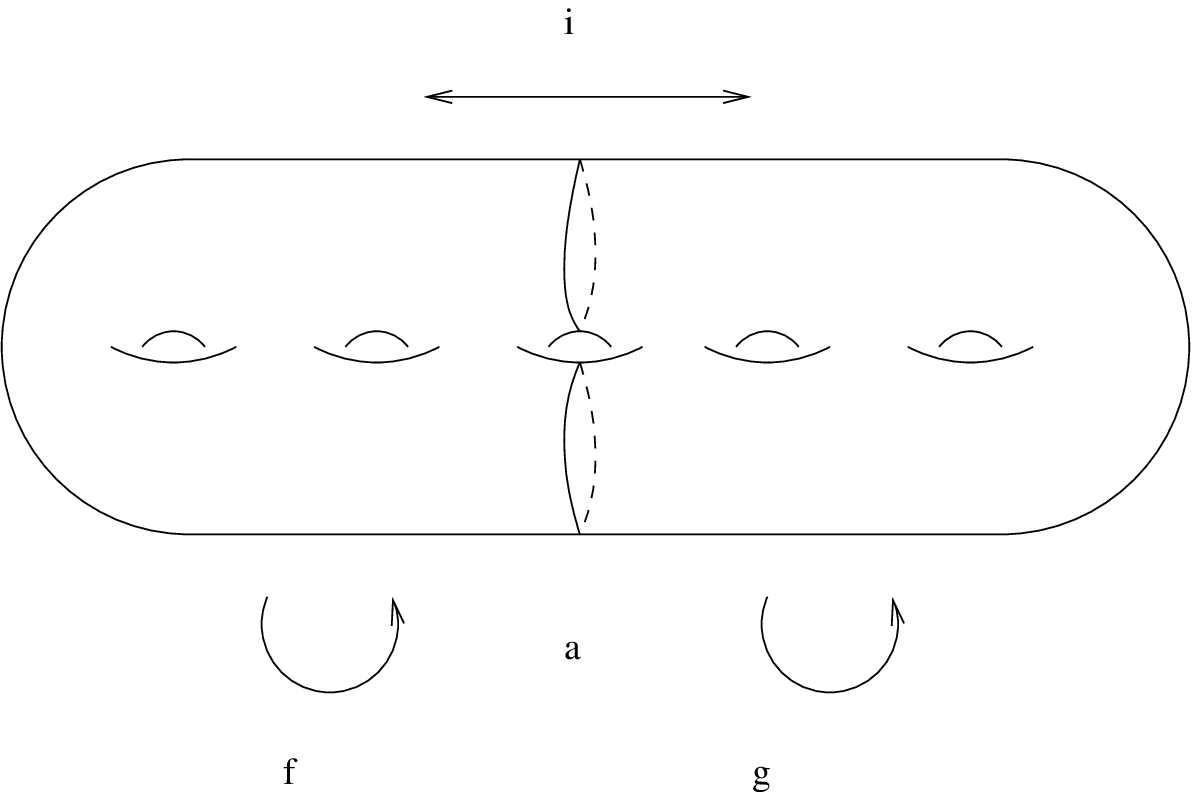}
\caption{\small {}}
\label{phi}
\end{figure}

A  component  of $S\setminus\lambda$ or of $S\setminus\mu$ may be of two types:
\begin{enumerate} 
\item \label{c1} an \emph{ideal polygon} (that is, a hyperbolic polygon with cusps, or, equivalently, a surface isometric to the convex hull in $\mathbb{H}^2$ of a finite number of points at infinity);
\item \label{c2} a \emph{crown} (that is, a hyperbolic annulus with cusps on one of its boundary components).
\end{enumerate}
  The two types of surfaces are represented in Figure \ref{crown}.

The components of $S\setminus\lambda$ or $S\setminus\mu$ that contain elements of the multicurve $\alpha$ can only be of the type (\ref{c2}).

  \begin{figure}[ht!]
\centering
\psfrag{1}{\small ideal polygon}
\psfrag{2}{\small crown}
\psfrag{b}{\small $\beta$}
\includegraphics[width=.60\linewidth]{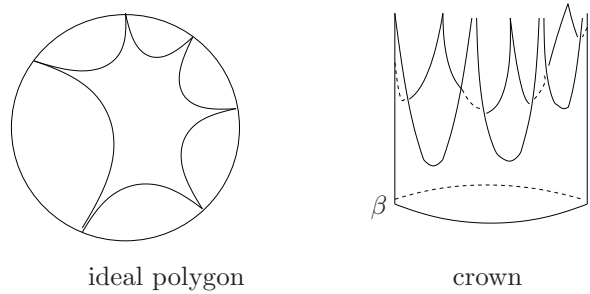}
\caption{\small {The two possible types of components of $S\setminus \lambda$ and $S\setminus \mu$.}}
\label{crown}
\end{figure}

  \begin{figure}[ht!]
\centering
\includegraphics[width=.60\linewidth]{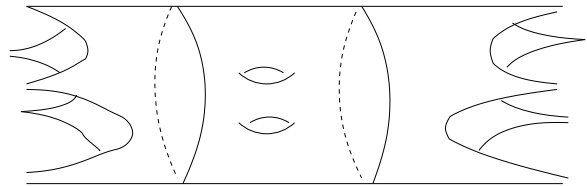}
\caption{\small {A crowned surface}}
\label{crowned}
\end{figure}

We consider a component $C$ of $S \setminus \mu$ and we lift it to the universal cover $\widetilde{S}$.  The map $\varphi$ also lifts to $\widetilde{S}$ and we can choose a lift $\widetilde{C}$ of $C$ such that it is fixed (setwise) by the lift $\widetilde{\varphi}$ of $\varphi$. The ideal points of $\widetilde{C}$ cut the circle at infinity into segments (Figure \ref{lift}). 

  \begin{figure}[ht!]
\centering
\psfrag{1}{\small }
\psfrag{2}{\small }
\psfrag{b}{\small $\widetilde{\beta}$}
\includegraphics[width=.70\linewidth]{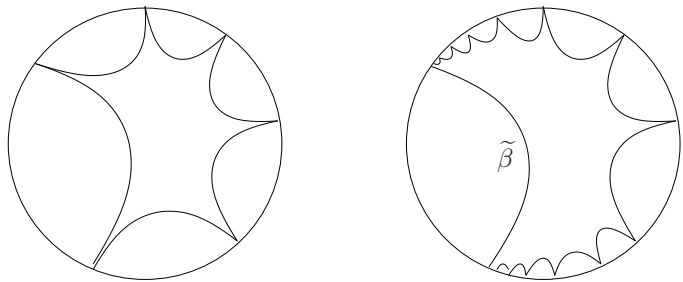}
\caption{\small {The lift of an ideal polygon (left) and of a crown (right). $\widetilde{\beta}$ is the lift of the boundary component $\beta$ of the crown.}}
\label{lift}
\end{figure}

Each segment between two ideal vertices of $\widetilde{C}$ contains in its interior exactly one fixed point; these fixed points of $\widetilde{\varphi}$ on $S^1_\infty$ are alternatively ideal vertices of lifts $\widetilde{C}_\mu$ of components  $C_\mu$ of $S\setminus\mu$ and ideal vertices of lifts $\widetilde{C}_\lambda$ of components $C_\lambda$ of $S\setminus \lambda$ and they are alternatively sinks and sources; cf. Thurston and Handel \cite{HT} and \cite{BC}. 

Consider now a completion $\overline{\mu}$ of $\mu$ such that all the extra leaves of $\overline{\mu}$ are isolated. Since $\varphi$ has no Dehn twists along its reducing multicurve, $\overline{\mu}$ is fixed by $\varphi$. The measured geodesic lamination $\lambda$ is totally transverse to $\overline{\mu}$, that is, $\lambda\in \mathcal{ML}(\overline{\mu})$. Let $F_\lambda$ be a measured foliation corresponding to $\lambda$. By a result of Thurston (cf. Theorem \ref{thm:thurston1} above), there exists a unique hyperbolic structure $h$ on $S$ such that 
$F_{\overline{\mu}}(h)=F_\lambda$, where $F_{\overline{\mu}}(h)$ is the horocyclic foliation.  
\begin{lemma} 
We have
\[F_{\overline{\mu}}(\varphi h)=\varphi(F_{\overline{\mu}}(h))=kF_{\overline{\mu}}(h).\]
\end{lemma}
\begin{proof}
Since $\varphi(\overline{\mu}=\overline{\mu})$ and since $h$ and $\varphi h$ are isometric hyperbolic structures on $S$, the leaves of the foliations $F_\lambda=F_{\overline{\mu}}(h)$ and those of $F_{\overline{\mu}}(\varphi h)$ have the same endpoints on the circle at infinity, once we lift the situation to the universal coverings. But these endpoints are those of $\lambda$ and they are fixed by $\widetilde{\varphi}$ setwise. Hence $F_{\overline{\mu}}(\varphi h)=F_\lambda$ (here,  equality forgets about the transverse measures), or, equivalently, $\lambda_{\overline{\mu}}(h)=\lambda_{\overline{\mu}}(\varphi h)=\lambda$ setwise.
Let $k'$ be the multiplicative factor:
\[\lambda_{\overline{\mu}}(h)=k'\lambda_{\overline{\mu}}(\varphi h)=k'\lambda.\]
Then,

\[
\ell_{\varphi h} =\ell_h(\varphi^{-1}\mu)=k \ell_h(\mu)  
=i(\lambda_{\overline{\mu}}(\varphi h),\mu)=k' i(\lambda_{\overline{\mu}}(h),\mu)=k' \ell_h(\mu).
\]

  \begin{figure}[ht!]
\centering
\psfrag{M}{\small $\widetilde{C}_{\mu}$}
\psfrag{L}{\small  $\widetilde{C}_{\lambda}$}
\psfrag{f}{\small $\widetilde{\varphi}$}
\psfrag{b}{\small $\widetilde{\beta}$}
\psfrag{1}{\small ideal polygon}
\psfrag{2}{\small crown}
\includegraphics[width=.70\linewidth]{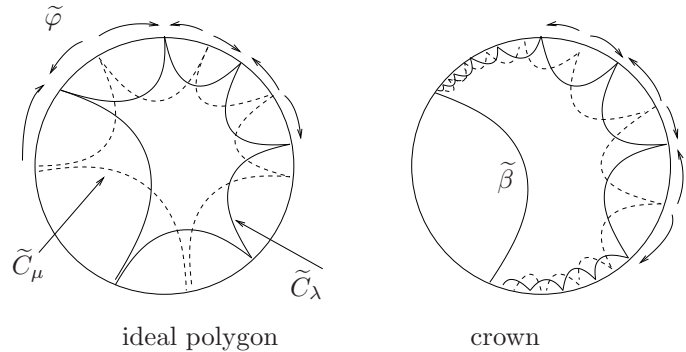}
\caption{\small {The dynamics on the universal cover.}}
\label{poly}
\end{figure}

Hence $k=k'$ and the lemma is proved.
\end{proof}

Thus, $\varphi$ fixes globally a stretch line and $d_T(h,\varphi h)=\log k$, that is, $\varphi$ acts by translation on that stretch line.

This proves proposition \ref{prop:ex}.
\end{proof}

 \end{document}